\pdfoutput=1
\documentclass[A4paper,8pt,conference]{IEEEtran}
\IEEEoverridecommandlockouts  
\usepackage{amsmath}
\usepackage{cite}
\usepackage{amsfonts}
\usepackage{amssymb}
\usepackage{graphicx}
\usepackage{epstopdf}
\usepackage{graphics}
\newtheorem{theorem}{Theorem}[section]
\newtheorem{lemma}[theorem]{Lemma}

\newenvironment{proof}[1][Proof]{\begin{trivlist}
\item[\hskip \labelsep {\bfseries #1}]}{\end{trivlist}}
\newenvironment{definition}[1][Definition]{\begin{trivlist}
\item[\hskip \labelsep {\bfseries #1}]}{\end{trivlist}}

\newcommand{\qed}{\nobreak \ifvmode \relax \else
      \ifdim\lastskip<1.5cm \hskip-\lastskip
      \hskip1.5cm plus0cm minus0.5cm \fi \nobreak
      \vrule height0.75cm width0.5cm depth0.25cm\fi}

\begin{document}
\title{A Geometric Approach to Rotor Failure Tolerant Trajectory Tracking Control Design for a Quadrotor }

\author{Ashutosh Simha$^1$, Sharvaree Vadgama$^1$ and Soumyendu Raha$^1$
\thanks{$^1$ Department of Computational and Data Sciences, Indian Institute of Science, Bangalore-12, India
}
}
\maketitle
\begin{abstract}
This paper addresses the problem of designing a trajectory tracking control law for a quadrotor UAV, subsequent to complete failure of a single rotor. 
The control design problem considers the reduced state space which excludes the angular velocity and orientation about the vertical body axis. The proposed controller enables the quadrotor to track the orientation of this axis, and consequently any prescribed position trajectory using only three rotors. The control design is carried out in two stages. First, in order to track the reduced attitude dynamics, a geometric controller with two input torques is designed on the Lie-Group $SO(3)$. This is then extended to $SE(3)$ by designing a saturation based feedback law, in order to track the center of mass position with bounded thrust. The control law for the complete dynamics achieves exponential tracking for all initial conditions lying in an open-dense subset. The novelty of the geometric control design is in its ability to effectively execute aggressive, global maneuvers despite complete loss of a rotor. Numerical simulations on models of a variable pitch and a conventional quadrotor have been presented to demonstrate the practical applicability of the control design. 
\end{abstract}
\section{Introduction}
Quadrotor Unmanned aerial vehicles (UAV) have been increasingly envisaged in defense, industrial and civil applications due to its simplified mechanical design and ease of maneuvering. The quadrotor consists of two pairs of symmetrically located counter rotating propeller blades which independently generate aerodynamic thrust along a common axis, in order to regulate the overall applied force and torques on the UAV. In typical flight missions, the independent rotor thrusts are regulated in tracking the center of mass position and heading angle of the UAV.
A basic understanding of the quadrotor dynamics and control design can be found in (\cite{beard}) and some recent research can be found in (\cite{andrea1, andrea2, andrea3, vijay_snap, vijay_avian, vijay2}). 
In conventional quadrotors, the thrust generated by each rotor is regulated by varying its speed. Such an actuation mechanism has a low control bandwidth due to saturation limits in the electro-mechanical circuit driving the rotor. Further, the rotor thrust needs to be strictly positive, thereby impairing the flight envelop. These factors have motivated the development of \textit{variable pitch} quadrotors (\cite{cutler1}, \cite{cutler2}) in which rotor thrust is regulated by varying the pitch angle of the propeller blades, while maintaining a constant rotor speed. This mechanism has a significantly higher actuation bandwidth than conventional rotors. Further, the blade pitch angles can be reversed to enable negative thrust generation. It has been shown in \cite{cutler3}, \cite{energies}, and \cite{kothari}, that the variable pitch mechanism appreciably enhances the flight envelop, thereby enabling aggressive maneuvers. 

The theoretical focus of this paper is to design a control law for a quadrotor, subsequent to complete failure of a single rotor.  With three functioning rotors, only a three dimensional submanifold of the output space can be completely regulated. A possible tracking solution is to relinquish control of the angular rate about the thrust axis, and choose the orientation of the thrust axis (i.e. reduced attitude) and net thrust as tracking outputs. 

The main new results in this paper over the existing studies, reviewed in {\it Related work}, are summarized as follows:
\begin{itemize}
\item The proposed control law can track globally defined reduced attitude and position trajectories, with only two control torques and a scalar thrust input.
\item The control law is free of singularities due to attitude parameterizations or input-output decoupling.
\item In the presence of bounded uncertainties, the tracking errors almost-globally converge at an exponential rate, to an arbitrarily small neighborhood of the origin.
\end{itemize}

\subsection{Related work}
While there is a significant amount of research in fault tolerant control of quadrotors with partial rotor loss, there are only a few results in case of complete rotor failure. In \cite{andrea} and \cite{andrearelaxed}, the authors present relaxed hover solutions with multiple rotor failures. The attitude dynamics are linearized about a hovering point where the yaw rate is a non-zero constant. In order to stabilize the position of the UAV, the orientation of the vertical axis (i.e. reduced attitude) and net rotor thrust is regulated. In \cite{landing1} and \cite{landing2}, a PID and back-stepping approach is used for emergency landing in case of rotor failure. In \cite{lanzonifac} and \cite{lanzonjgcd}, the authors present a hierarchical control design in which the inner loop controls the reduced attitude and the outer loop controls the position. The inner loop consists of a robust feedback linearization based controller and the outer loop is a  $H_{\infty}$ based controller for the translational dynamics, linearized about a hover point. In \cite{peng} and \cite{akhtar}, the authors present static and dynamic feedback linearization based controllers to regulate the reduced attitude and position of the quadrotor. Control designs based on small angle or linear approximations restrict the motion of the quadrotor to near-hover maneuvers. Further, feedback linearization based control laws mentioned above, encounter singularities when the roll and pitch angles are $\pi/2$ or when the net thrust is zero. In order to avoid this, the initial state errors have to be restricted within a sufficiently small neighborhood of the origin. These factors render the existing fault tolerant control designs ineffective in tracking global trajectories or performing aggressive maneuvers (such as attitude recovery from an inverted pose). It is imperative to understand that post rotor failure, the orientation of the quadrotor may undergo large deviations from the operating point, thereby necessitating global maneuvering capability. 

In recent times, globally stabilizing geometric controllers which exploit the intrinsic structure of the underlying manifold, have been developed. Here, singularities due to attitude parameterizations or input-output decoupling are avoided. Approaches such as that in \cite{maithri} and \cite{bullo} stabilize mechanical systems on Lie Groups using nonlinear proportional-derivative (PD) control. One such pioneering control design for quadrotors on the Lie Group $SE(3)$ has been presented in \cite{tlee} and \cite{tlee_tac}.  Reduced attitude stabilization to a fixed point on $S^2$  with two control torques is presented in \cite{bulloreduced}. In \cite{sphere}, the primary axis of a rigid body on $S^2$ as well as the angular velocity about this axis, are tracked using three independent torques. In \cite{teel}, global reduced attitude tracking with three torques is achieved by constructing a synergistic family of potential functions on $S^2$. To the best of our knowledge, none of the above mentioned control laws are suitable for trajectory tracking with three functioning rotors (i.e. two torque inputs and a net thrust).

\subsection{Proposed control design}
First, a control law is developed on $SO(3)$ in order to track a commanded reduced attitude trajectory. A back-stepping feedback law for the two torques about the horizontal body axes of the quadrotor is designed based on the geometric structure of $TS^2$ (on which the reduced attitude dynamics evolve). The back-stepping law is preferred over standard geometric PD controllers as they are not applicable in case of reduced dimensional input space. Subsequently, a saturation based feedback law is designed for the translational dynamics, in order to track a prescribed position trajectory with bounded thrust. This also ensures that the commanded thrust vector does not vanish, thereby ensuring that the reference reduced attitude trajectory is well defined. The control law is further robustified in order to account for propeller-induced gyroscopic moment (which is typically neglected when all four rotors are functional), and rotational drag. The tracking errors are shown to exponentially converge to an arbitrarily small open neighborhood of the origin. The performance of the controller is first demonstrated through simulations on a variable pitch quadrotor which is capable of negative thrust generation. Then, the same control law is simulated with a strictly positive rotor thrust constraint, in order to demonstrate its effectiveness on a conventional quadrotor. In this case however, aggressive trajectory tracking is successful provided that the angular velocity about the vertical axis is high enough. 

The paper is organized as follows. In section 2., the nonlinear dynamics of the quadrotor is presented. Section 3. contains the formulation of the geometric control law. In section 4., simulation results with the proposed control law have been presented, which is then followed by concluding remarks. 

\section{Problem Formulation}
\begin{figure}[h]
 \includegraphics[width=0.4\textwidth]{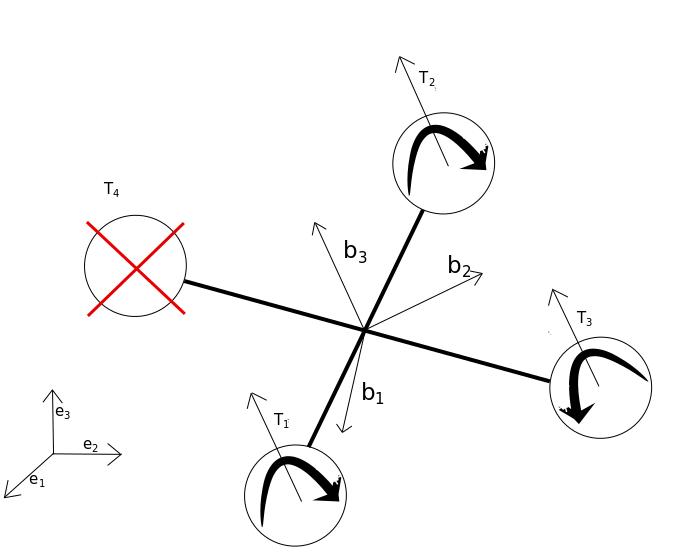}
\caption{Quadrotor model}
\label{fig:quad}
\end{figure} 
\subsection{Quadrotor Dynamics}
Consider the quadrotor as shown in Fig.\ref{fig:quad}.
 Let $\{\vec{e}_1,\vec{e}_2,\vec{e}_3\}$ denote the inertial frame and $\{\vec{b}_1,\vec{b}_2,\vec{b}_3\}$ denote the body frame. The four identical rotors are designed to generate thrusts $T_1,~T_2,~T_3,~T_4$ along $\vec{b}_3$. In this paper, it is assumed that the fourth rotor has been completely disabled post fault detection (i.e. $T_4\equiv 0$). The origin of the body frame is located at the center of mass. $R\in SO(3)$ is the rotation matrix from the $\vec{b}$ frame to $\vec{e}$ frame, denoting the attitude of the quadrotor. $\Omega$ is the angular velocity in the body frame. $x$ and $v$ denote the position and velocity of the center of mass. $m$ denotes the mass of the quadrotor, $J=diag(J_1,J_2,J_3)$ is the moment of inertia matrix in the body frame, $J_r$ is the inertia of the rotors, and $\tau_d$ is the aerodynamic rotational drag. The total thrust and torque due to the rotors is represented by $f$ and $M$ respectively. The first and second rotor spin clockwise and the third spins anti-clockwise at angular speeds $\omega_1,\omega_2,\omega_3$ respectively. The rigid body equations of motion are derived using the \textit{Euler Poincar\`e} formalism on $SE(3)$ (the configuration manifold of the quadrotor) as follows:
 \begin{equation}
 \begin{matrix}
 \dot{x}=v \\
 m\dot{v}=-mge_3+fRe_3 \\
 \dot{R}=R\hat{\Omega}\\
  J\dot{\Omega}=J\Omega\times\Omega+g_r-\tau_d+M
    \end{matrix}
 \label{quad}
 \end{equation}where $\hat{.}:\mathbb{R}^3 \to \mathfrak{so}(3)$ is defined as $\hat{x}y=x\times y,~x,y\in\mathbb{R}^3$. We will denote $(.)^\vee$ as its inverse map throughout the paper and $e_i$ as the representation of $\vec{e}_i$. 
 
 When all four rotors are functioning, the gyroscopic moment $g_r$ generated by their different rotational speeds is typically neglected. However, $g_r$ may be significant in case of rotor failure and is therefore included. This effect has been modeled as an additional moment by the following equation (\cite{gr}). 
 \begin{equation}
 g_r=J_r(\Omega\times e_3)(\omega_1-\omega_2+\omega_3)
 \label{g_r}
 \end{equation}where $\omega_i$ is the speed of the $i^th$ rotor.
 
 The aerodynamic drag torque $\tau_d$ is hard to model as it depends on the quadrotor profile. The model used here has been adopted from \cite{andrearelaxed}, and is based on the form drag of a translating object (\cite{cormick}) which is quadratic in the vehicle's angular velocity as:
 \begin{equation}
 \tau_d=||\Omega||K_d\Omega,
 \label{tau_d}
 \end{equation}where $K_d$ is a positive-definite matrix. 
 
 Let $T_i$ and $D_i$ denote the thrust and drag induced torque generated by the $i^{th}$ rotor. $f$ and $M$ are then obtained for the 'X' configuration of the quadrotor as:
 \begin{eqnarray}
 f&=&T_1+T_2+T_3    \nonumber  \\ 
 M_1&=&d(T_1 -T_2-T_3)  \nonumber  \\ 
 M_1&=&d(T_1+T_2-T_3) \nonumber  \\  
 M_3&=&D_1-D_2+D_3
\label{f,M}
 \end{eqnarray}

In case of a variable pitch quadrotor, the thrust in each rotor is varied by changing the collective blade pitch angle of the propellers, while maintaining constant rotor speed. The relation between rotor thrust and drag, with the blade pitch angle has been adopted from \cite{cutler1} as: 
\begin{eqnarray}
T_i&=&b_L\omega_i^2\gamma_i \nonumber \\ 
D_i&=& b_{D_1}\omega_i^2+b_{D_2}\omega_i^2\gamma_i^2+b_{D_3}\omega_i\gamma_i
 \label{TiandDi}
  \end{eqnarray}
where $\gamma_i$ is the pitch angle of the $i^th$ rotor, and $b_L,~b_{D_1},~b_{D_2},~b_{D_3}$ are constants which depend on the aerodynamic profile of the propeller. Though in principle one can use the rotor speed as an additional control input, this is not advisable due to significant aerodynamic uncertainties when $\omega_i$ is low or rapidly fluctuating.

\section{Geometric Control Design}
 \begin{figure}[h]
 \includegraphics[width=0.5\textwidth]{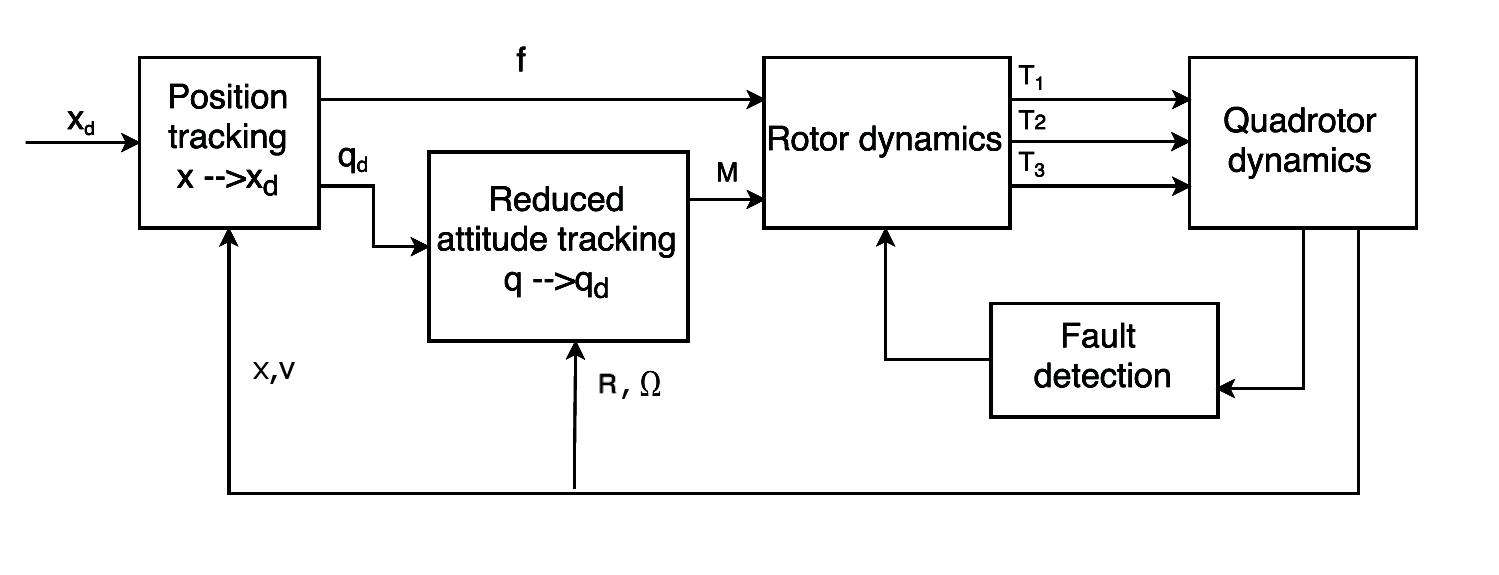}
\caption{Controller Structure}
\label{fig:block}
\end{figure} 
With three functioning rotors, the control inputs are chosen as $f$ and $U=[U_1,U_2]^T=[M_1/J_1,M_2/J_2]^T$. Control of the angular velocity about the vertical axis is relinquished. First, a reduced attitude controller with input $U$ is designed on $S^2$, which is then extended to a trajectory tracking controller on $S^2\times \mathbb{R}^3$ by designing a control law for $f$. 
 \subsection{Reduced Attitude Tracking Controller on $SO(3)$}
The reduced attitude of the quadrotor is defined as the pointing direction of the thrust axis (i.e. body $z$ axis) which is obtained via the projection $\pi:SO(3)\to \mathbb{S}^2$ (as introduced in \cite{bulloreduced}), defined as
\begin{equation}
\pi(R)=Re_3
\end{equation}Let $q=Re_3$ denote the reduced attitude , and $\vec{r}_1 =Re_1$, $\vec{r}_2=Re_2$ denote the two horizontal body axes. The dynamics of $q$ can be obtained from (\ref{quad}) as:
\begin{eqnarray}
\dot{q}&=&\Omega_1\vec{r}_2-\Omega_2\vec{r}_1  \nonumber \\
\dot{\Omega}_1&=&\frac{(J_2-J_3)}{J_1}\Omega_2\Omega_3+U_1+d_1(\Omega) \nonumber \\ 
\dot{\Omega}_2&=&\frac{(J_3-J_2)}{J_2}\Omega_1\Omega_3+U_2+d_2(\Omega) 
\label{q_dynamics}
\end{eqnarray}Here, $d=[d_1,d_2]^T$ denotes the uncertainties due to propeller induced gyroscopic torque and rotational drag. From their respective forms as described in (\ref{g_r}) and (\ref{tau_d}), we can bound $d$ as:
\begin{equation}
||d||_2\leq(||\tilde{\Omega}||_2+||\tilde{\Omega}||^2_2)\Delta
 \label{d_bound}
 \end{equation}where $\Delta=max~\{J_r\omega_1,\lambda_{max}(K_d) \}$ and $\tilde{\Omega}=[\Omega_1,\Omega_2]^T$. Let $q_d(t)\in \mathbb{S}^2$ denote the reference reduced attitude trajectory. Motivated by \cite{leeso3} and \cite{sphere}, we define a reduced attitude error function as:
\begin{equation}
\Psi(q,q_d)=2-\dfrac{2}{\sqrt{2}}\sqrt{1+q_d^Tq}
\label{psi}
\end{equation}This choice is motivated by the fact that the left trivialized differential of the error function does not vanish when $q\to-q_d$. However, this happens in case of the conventional error function on $\mathbb{S}^2$ as defined in \cite{bullo}, thereby rendering the tracking performance poor. It can be observed that the error function satisfies:

\begin{equation}
\begin{matrix}
\Psi(q,q_d)\in [0,2],~\\ \\
\Psi(q,q_d)=0\iff q=q_d
\end{matrix}
\end{equation}

With $q_d$ constant, the differential of $\Psi(q,q_d)$ along $T_q^*\mathbb{S}^2$ is computed via the tangent map corresponding to the projection
  
   $\pi_{\mathbb{S}^2}:\mathbb{R}^3-\{\vec{0}\} \to \mathbb{S}^2$ where $\pi_{\mathbb{S}^2}(x)=x/||x||_2$, as
\begin{equation}
d_1\Psi_{\mathbb{S}^2}(q,qd)=\dfrac{1}{\sqrt{2}\sqrt{1+q_d^Tq}}q\times (q\times q_d)
  \label{d_psi}
  \end{equation}We denote the reduced attitude error vector as:
  \begin{equation}
   e_q:= d_1\Psi_{\mathbb{S}^2}(q,qd)
   \end{equation}  Note that this quantity is well defined as long as $q_d^Tq>-1$ i.e. when the angle between them is less than $180^{\circ}$. We now establish the following relation between $e_q$ and $\Psi$.
  
  \begin{lemma}  
  \begin{equation}
  ||e_q||_2^2\leq \Psi \leq 2  ||e_q||_2^2,~\forall (q,q_d)\in \Psi^{-1}[0,2)
  \end{equation}
  \label{eq_leq_psi}
 \end{lemma}
 \begin{proof}
 By using the identity: 
 
 $||q_d^Tq||_2=\cos(\theta)$ and $||q\times(q\times q_d)||_2=\sin(\theta)$, 
 
 where $\theta$ denotes the angle between $q$ and $q_d$. 
 \end{proof}
 Next, in order to define the velocity error on $T_{\mathbb{S}^2}$, we define the transport map $\mathcal{T}_{\mathbb{S}^2}(q,q_d):T_{q_d}\mathbb{S}^2\to T_{q}\mathbb{S}^2$ (\cite{bullo}) as follows.
 \begin{equation}
 \mathcal{T}_{\mathbb{S}^2}(q,q_d).v=(q_d\times v)\times q,~\forall v\in T_{q_d}\mathbb{S}^2
 \label{Tau}
 \end{equation}We make the following observation:
 \begin{lemma}
 The pull-back of the transport map $\mathcal{T}$ satisfies the equation:
  \begin{equation}
    \mathcal{T}_{\mathbb{S}^2}(q,q_d)^*(d_1\Psi_{\mathbb{S}^2}(q,qd))=-d_2\Psi_{\mathbb{S}^2}(q,qd)  
   \end{equation}
  \label{Tau_lemma}
 \end{lemma}
 \begin{proof}
 By using the identity 
 
 $(x\times y)\times z=(z^Tx)y-(z^Ty)x,~\forall x,y,z\in\mathbb{R}^3$.
\end{proof}  
 
 We now define the velocity error vector as:
 \begin{equation}
 e_{\dot{q}}:=\dot{q}-\mathcal{T}_{\mathbb{S}^2}(q,q_d).\dot{q_d}
 \label{eqdot}
 \end{equation}
 The derivative of the error function can be obtained using Lemma \ref{Tau_lemma} as
 \begin{eqnarray}
 \frac{d}{dt}\Psi(q,q_d)&=&d_1\Psi_{\mathbb{S}^2}(q,q_d)\dot{q}+d_2\Psi_{\mathbb{S}^2}(q,q_d)\dot{q}_d \nonumber \\ 
 &=&d_1\Psi_{\mathbb{S}^2}(q,q_d)e_{\dot{q}}
 \label{Psi_dot}
  \end{eqnarray}In order to stabilize the dynamics of $\Psi$, we require that $e_{\dot{q}}$ satisfies:
 \begin{equation}
 e_{\dot{q}}=-k_q d_1\Psi_{\mathbb{S}^2},~k_q>0
 \label{eqdot1}
 \end{equation} Then, $\dot{\Psi}$ can be obtained as 
 \begin{equation}
 \dot{\Psi}=-k_q||e_q||_2^2
 \end{equation}Further, Lemma \ref{eq_leq_psi} asserts that $\Psi$ can be sandwiched between two positive definite quadratic forms in $e_q$, thereby ensuring that the dynamics of $\Psi$ can be bounded as:
  \begin{equation}
  \Psi(q(t),q_d(t))\leq \Psi(q(0),q_d(0))e^{-k_qt}
  \end{equation}
  
Equation (\ref{eqdot1}) can be written as:
 \begin{equation}
 \Omega_1\vec{r}_2-\Omega_2\vec{r}_1=\mathcal{T}_{\mathbb{S}^2}(q,q_d).\dot{q_d}-k_qd_1\Psi_{\mathbb{S}^2}
 \end{equation}Since $span~\{\vec{r}_1,\vec{r}_2\}=T_q\mathbb{S}^2$, the above equation admits a unique solution for $\Omega_1$ and $\Omega_2$ which is given by:

 \begin{equation}
\Omega_d=\begin{bmatrix}
\langle\vec{r}_2,(\mathcal{T}_{\mathbb{S}^2}(q,q_d).\dot{q_d}-k_qd_1\Psi_{\mathbb{S}^2})\rangle\\
\langle-\vec{r}_1,(\mathcal{T}_{\mathbb{S}^2}(q,q_d).\dot{q_d}-k_qd_1\Psi_{\mathbb{S}^2})\rangle
 \end{bmatrix}
 \label{wd}
 \end{equation}
 
 We now construct a control law in order to track a commanded reduced attitude trajectory. 
 
Define:
\begin{eqnarray*}
e_{\Omega}&:=&([\Omega_1,\Omega_2]^T-\Omega_d), \\
f_J(\Omega)&:=&\bigg[\frac{(J_2-J_3)}{J_1}\Omega_2\Omega_3,\frac{(J_3-J_1)}{J_2}\Omega_1\Omega_3\bigg]^T 
\label{definitions_theorem1}
\end{eqnarray*} 
   
 and 
\begin{equation}
U_{\Delta}=\left\lbrace \begin{matrix} \dfrac{e_{\Omega}}{||e_{\Omega}||_2},~ ||e_{\Omega}||_2 >tol \\ \\
\dfrac{e_{\Omega}}{tol},~ ~||e_{\Omega}||_2 \leq tol
\end{matrix} \right\rbrace
\label{U_delta}
\end{equation} where $tol>0$ is a positive constant depending on the slew rate of the torque actuation.

 \begin{theorem}
 Given a reference trajectory $q_d(t)$ which is smooth with bounded derivatives, the control law:
 \begin{eqnarray}
   U(R,\Omega,t):=&&-\alpha\begin{bmatrix}
  \langle d_1\Psi_{\mathbb{S}^2},\vec{r}_2\rangle\\
  \langle d_1\Psi_{\mathbb{S}^2},-\vec{r}_1\rangle 
  \end{bmatrix}- k_{\Omega}e_{\Omega}-f_J(\Omega)\nonumber \\ \nonumber \\  &&+\dot{\Omega}_d-U_{\Delta}(||\tilde{\Omega}||_2+||\tilde{\Omega}||^2_2)\Delta, 
   \label{U}
  \end{eqnarray} ensures that $e_q$ and $e_{\Omega}$ exponentially converge to an arbitrarily small open neighborhood of the origin, for all initial conditions in the open-dense sublevel set $\Psi^{-1}[0,2)$ satisfying:
 \begin{equation}
 ||e_{\Omega}(0)||_2<2\alpha(2-\Psi(0)).
 \label{condition theorem1}
 \end{equation}
 
 Further, the sublevel set $\Psi^{-1}[0,2)$ remains invariant for the closed loop flow of (\ref{q_dynamics}) with the feedback law (\ref{U}). 
 \label{theorem1}
 \end{theorem}
 \begin{proof}
 Consider the Lyapunov function: 
 \begin{equation}
 V_1:=\alpha\Psi+\frac{1}{2}||e_{\Omega}||_2^2.
 \label{v1}
 \end{equation}
 Its derivative along the trajectories of (\ref{q_dynamics}) with the control law (\ref{U}), is obtained using (\ref{Psi_dot}), (\ref{wd}), (\ref{definitions_theorem1}), and (\ref{U_delta}), as:
 \begin{eqnarray}
  \dot{V}_1=&&-\alpha k_q||e_q||_2^2- k_{\Omega}||e_{\Omega}||_2^2 \\ \nonumber &&+\langle e_{\Omega},d-U_{\Delta}(||\tilde{\Omega}||_2+||\tilde{\Omega}||^2_2)\Delta\rangle
  \end{eqnarray}When $||e_{\Omega}||_2>tol$, the inner product term can be shown to be negative using (\ref{d_bound}) along with the Cauchy-Schwarz inequality. When $||e_{\Omega}||_2\leq tol$, a straightforward calculation shows that we can bound the inner product term using (\ref{d_bound}), as:
  
 \begin{equation}
 \langle e_{\Omega},d-U_{\Delta}(||\tilde{\Omega}||_2+||\tilde{\Omega}||^2_2)\Delta\rangle~~\leq tol\Delta(||\tilde{\Omega}||_2+||\tilde{\Omega}||^2_2)
 \end{equation}
 
 Further, since $\dot{q}_d$ is bounded, $\Omega_d$ can be uniformly bounded. When $||e_{\Omega}||_2<tol$, one can show that $||\tilde{\Omega}||_2$ is uniformly bounded within a neighborhood of $\Omega_d$, using the triangle inequality. Therefore, by appropriately selecting the value of $tol$, the inner product term can be bounded above by an arbitrarily chosen constant $\epsilon>0$ as: 
 
 \begin{equation}
 \langle e_{\Omega},d-U_{\Delta}(||\tilde{\Omega}||_2+||\tilde{\Omega}||^2_2)\Delta\rangle~~\leq \epsilon
 \end{equation}
 
 With this, the derivative of $V_1$ can be bounded as:
 \begin{equation}
 \dot{V}_1\leq -\alpha k_q||e_q||_2^2- k_{\Omega}||e_{\Omega}||_2^2+\epsilon
 \label{v1dotfinal}
 \end{equation}From (\ref{eq_leq_psi}) we have,
 \begin{equation}
 -\alpha k_q||e_q||_2^2\leq-\dfrac{\alpha k_q}{2}\Psi
 \end{equation}therefore,
 \begin{equation}
 \dot{V}_1\leq -\beta V_1+\epsilon
 \end{equation}where $\beta=min\big(\frac{k_q}{2},k_{\Omega}\big)$.
 Further, since $\epsilon$ can be arbitrarily defined, $V_1$ can be guaranteed to be strictly monotonically decreasing in $\Psi^{-1}(\bar{\epsilon},2)$ where $\bar{\epsilon}$ can be made arbitrarily small. In this region, 
 \begin{equation}
\alpha \Psi(t)\leq V_1(0)\leq \alpha \Psi(0)+\dfrac{1}{2}||e_{\Omega}(0) ||_2^2
 \end{equation}Therefore, applying the condition (\ref{condition theorem1}), we obtain:
 \begin{equation}
 \alpha\Psi(t)\leq 2\alpha,~\forall \Psi(0)\in(\bar{\epsilon},2),
 \end{equation}Since $\bar{\epsilon}$ was arbitrary, the sublevel set $\Psi^{-1}[0,2)$ remains invariant.
 \end{proof}

 \textbf{Remark 1:}
 In this control design, the four tuning parameters are $k_q,~k_{\Omega},~\alpha,~\Delta$. A higher value of $\alpha$ ensures that the control law does not encounter any discontinuities when $\Psi=2$. Such a high gain may be necessary when the initial angular velocity error is large. $k_q,~k_{\Omega}$ may be chosen to arbitrarily dictate the rate of exponential tracking. Finally, choosing a higher value of $\Delta$ ensures a tighter bound on the asymptotic tracking error. 
 
  \textbf{Remark 2:}
In case a high gain $\alpha$ is not admissible, a possible solution is to mollify the error function $\Psi$ such that $d_1\Psi$ is continuous at $\Psi=2$. For example, one such mollification is the standard error function $\Psi=1-q_d^Tq$. With the same form of control as in (\ref{U}), the derivative of the Lyapunov function $V_1$ is obtained as in (\ref{v1dotfinal}). However in this case, $\dot{V}_1$ may vanish when $\Psi=2$ and $e_{\Omega}=0$. The Lasalle-Yoshizawa theorem (\cite{krstic}) can now be applied to conclude that the limit set of the trajectories is $e_q=0,~e_{\Omega}=0$. It can be seen that $e_q$ may vanish when $q=q_d$ or $q=-q_d$. It is then necessary to show that the undesired equilibrium point $q=-q_d$ is locally unstable (atleast when $\epsilon\approx 0$). 

Consider a function 
$W=2\alpha-V_1$ which vanishes when $q=-q_d$. From the continuity of $\Psi$, it can be shown that in any arbitrarily small neighborhood of $(q,e_{\Omega})=(-q_d,0)$, there exists points $q$ where $2-\Psi>0$. At such points, when $e_{\Omega}$ is small enough, it can be shown that $W>0$. Further, in an open neighborhood of the undesired equilibrium point (excluding it), $\dot{W}=-\dot{V}>0$. Since the complement set of the equilibria is positively invariant, Chetaev's theorem (\cite{khalil}) can be applied to conclude that the undesired equilibria are unstable. Hence, the trajectories of the system converge asymptotically to the stable equilibrium $(\Psi,e_{\Omega})=(0,0)$, for almost all initial conditions. 

Note however, that such analysis may not be valid when $\epsilon$ is significant, thereby further justifying our choice of error function. 
  
 \subsection{Position Tracking Controller on $SE(3)$}
 Let $x_d$ denote a smooth reference trajectory for the position of the center of mass. We assume that $x_d$ and its derivatives are bounded. Let $e_x:=x-x_d$ and $e_v:=\dot{x}-\dot{x}_d$ denote the position and velocity errors. We now design a saturation based feedback law in order to track the position trajectory with bounded thrust. 
 
 \begin{definition}[Definition:]
 Given constants $a$ and $b$ such that $0<a\leq b$, a function $\sigma:\mathbb{R}\to\mathbb{R}$ is said to be a smooth linear saturation function with limits $(a,b)$, if it is smooth and satisfies:
 \begin{enumerate}
 \item $s\sigma(s)>0,~\forall s\neq 0$
 \item $\sigma(s)=s,~\forall |s|\leq a$
 \item $|\sigma(s)|\leq b,~\forall s\in \mathbb{R}$
 \end{enumerate}
 \end{definition}It is well known that such smooth saturation functions exist. For example, consider the integral of a smooth function with compact support, which is constant within a sub-interval of its support (\cite{shakarchi}). Such a function when shifted by a constant, satisfies the conditions in the definition. In practice, one can approximate these functions using polynomials.
 
 Let $\sigma_1$ and $\sigma_2$ be two saturation functions with limits $(a_1,b_1)$ and $(a_2,b_2)$ such that,
\begin{equation}
b_1<\frac{a_2}{2}.
\label{a1b1}
  \end{equation}

 We now define a control law for $\hat{f}$ which is the total vector thrust acting on the rigid body, as follows:
 \begin{equation}
 \hat{f}=\bar{\sigma}(e_x,e_v)+f_d
 \label{fhat}
 \end{equation}
 where
  \begin{equation}
 f_d=m\ddot{x}_d+mge_3,
 \end{equation}

  \begin{equation}
 \bar{\sigma}(e_x,e_v)=-\begin{bmatrix}
 &&\sigma_2\bigg(\frac{k_1}{k_2}e_{v_1}+\sigma_1\bigg(k_2me_{v_1}+k_1e_{x_1}\bigg)\bigg ) \\ \\
 &&\sigma_2\bigg(\frac{k_1}{k_2}e_{v_2}+\sigma_1\bigg(k_2me_{v_2}+k_1e_{x_2}\bigg)\bigg ) \\ \\
 &&\sigma_2\bigg(\frac{k_1}{k_2}e_{v_3}+\sigma_1\bigg(k_2me_{v_3}+k_1e_{x_4}\bigg)\bigg ) 
 \end{bmatrix},
 \label{sigma_bar}
 \end{equation} and  $k_1,~k_2$ are positive constants.
 
 When $fRe_3=\hat{f}$, it can be established from Theorem 2.1 in \cite{teelglobal} that the tracking errors enter the linear region of the saturation functions in finite time, and remain within thereafter. This would ensure that the origin of the tracking errors is exponentially attractive. 
 
 The following Lemma will be subsequently used to demonstrate that the tracking errors enter the linear region in finite time, when the reduced attitude error is sufficiently bounded.
 
 \begin{lemma}
 Let $\sigma_1$ and $\sigma_2$ be saturation functions with limits as prescribed in (\ref{a1b1}). Then, the trajectories of the system
 \begin{eqnarray*}
 \dot{y}_1&=&y_2 \nonumber \\
 m\dot{y}_2&=&-\sigma_2((k_1/k_2)y_2+\sigma_1(k_1y_1+k_2my_2))+\xi(t),
 \end{eqnarray*}enter the linear region of $\sigma_1$ and $\sigma_2$ in a finite time $t_2$ and remain within thereafter if,
  
  $|\xi(t)|<min~((a_2/2)-b_1,a_1),~\forall t>0$. 
 \label{lemma_sigma}
 \end{lemma}
 
 \begin{proof}
 Let $w_1=my_2^2$. We obtain,
 
 \begin{equation}
 \dot{w}_1=2y_2(-\sigma_2((k_1/k_2)y_2+\sigma_1(k_1y_1+k_2my_2))+\xi(t))
 \end{equation} 
 
 When $|y_2|\geq(k_2/k_1)(a_2/2)$, using the bound on $\xi$ and (\ref{a1b1}), we can see that $\dot{w}_1$ is uniformly negative definite.
 
Hence,  $\exists t_1>0,~ \ni |y_2(t)|<(k_2/k_1)(a_2/2),~ \forall t>t_1$. 
 
 Using the bound on $b_1$, we conclude that $\sigma_2$ operates in its linear region after $t_1$. 
 
Let $w_2=||(k_1y_1+k_2my_2)||_2^2$. When $t>t_1$, its derivative is obtained as,

\begin{equation}
\dot{w}_2=-2(k_1y_1+k_2my_2)(\sigma_1(k_1y_1+k_2my_2)+\xi(t))
\end{equation}

From the definition of $\sigma_1$ and the bound on $\xi$, $w_2$ is uniformly negative definite when $|k_1y_1+k_2my_2|\geq a_1$. 

Hence, $\exists t_2>t_1>0,~ \ni |k_1y_1+k_2my_2|<a_1,~ \forall t>t_2$. 

It can there be concluded that $\sigma_1$ and $\sigma_2$ operate in their respective linear regions after $t_2$. 
 \end{proof}

We now define the commanded reduced attitude trajectory as:
 \begin{equation}
 q_d=\dfrac{\hat{f}}{||\hat{f} ||_2}
 \label{qd_command}
 \end{equation}
 
 This is well defined when $||\hat{f} ||_2$ is bounded away from zero. One way to ensure this is to choose a bound on $b_2$ as:
 
 \begin{equation*}
~b_2<\inf\limits_{t>0}\{||f_d(t) ||_{\infty}\}.
 \end{equation*}
 
 The control law for the net thrust $f$ is then chosen as:
 \begin{equation}
 f=~\langle\hat{f},Re_3\rangle.
 \label{f}
 \end{equation}

 \begin{theorem}
 Consider the control law for $U$ and $f$ as given in (\ref{U}) and (\ref{f}) such that the condition (\ref{condition theorem1}) is satisfied. Further, define the matrices:
 \begin{eqnarray}
 W_1&=&\begin{bmatrix}
 \frac{ck_x}{m}(1-\sin(\theta_0)) & -\frac{ck_v}{2m}(1+\sin(\theta_0)) \\
 -\frac{ck_v}{2m}(1+\sin(\theta_0)) & k_v(1-\sin(\theta_0))-c,
  \end{bmatrix},\nonumber
 \\\nonumber
 \\
 W_2&=&2\begin{bmatrix}
 (c/m)||f_d ||_2 & 0 \\ \\  a_1+||f_d ||_2 & 0 
 \end{bmatrix}.
 \label{wmatrices}
\end{eqnarray} 


 Given $0<k_x:=k_1$, $0<k_v:=(k_1/k_2)+k_2$, and $\theta_0<\pi/2$, we choose  positive constants $c$, $k_q$, $k_{\Omega}$, such that 
\begin{small}
\begin{eqnarray}
&&c<min\bigg\{ k_xk_v(1-\sin(\theta_0))^2\bigg(k_x(1-\sin(\theta_0)) \nonumber \\  &&+\frac{k_v^2(1+\sin(\theta_0))^2}{4m} \bigg)^{-1}, k_v(1-\sin(\theta_0)), \sqrt{k_x/m}~\bigg\},~  \nonumber \\
&&min~(\alpha k_q,k_\Omega)>\frac{4||W_2 ||^2}{\lambda_{min}(W_1)}.
\label{condition on theorem2}
\end{eqnarray}
\end{small}
Then, the tracking errors $e_x$, $e_v$, $e_q$, $e_{\Omega}$, exponentially converge to an arbitrarily small open neighborhood of the origin, for all initial conditions lying in an open-dense subset. 
 \end{theorem}
 \begin{proof}
 Assuming that $\theta<\frac{\pi}{2}$, the dynamics of $e_v$ can be written using (\ref{quad}) as:
 \begin{equation}
 m\dot{e}_v=-mge_3-m\ddot{x}_d+\dfrac{f}{q_d^Tq}q_d+\mathcal{E}
 \end{equation}where $\mathcal{E}\in \mathbb{R}^3$ is defined as
 \begin{equation}
 \mathcal{E}=f\bigg(q- \dfrac{q_d}{q_d^Tq} \bigg)
 \label{E}
 \end{equation}Further, we can write
 \begin{equation}
 \dfrac{f}{q_d^Tq}q_d=\dfrac{||\hat{f}||_2q_d^Tq}{q_d^Tq}.\dfrac{\hat{f}}{||\hat{f}||_2}=\hat{f}
 \end{equation}Hence, using (\ref{fhat}) we can write:
 \begin{equation}
 m\dot{e}_v=\bar{\sigma}(e_x,e_v)+\mathcal{E}
 \label{evdot}
 \end{equation}
 
From (\ref{E}), we can bound $\mathcal{E}$  as
\begin{equation}
\mathcal{E}\leq ||\hat{f}||_2||((q_d^Tq)q-q_d)||_2
\end{equation}

where the term $||((q_d^Tq)q-q_d)||_2 =|\sin(\theta)|$. Further, the commanded thrust $\hat{f}$ can be bounded using the saturation limit $b_2$ as
\begin{equation}
||\hat{f}||_2\leq\sqrt{3}b_2+\sup\limits_{t>0}\{||f_d(t)||_2\}:=B
\end{equation}

From Theorem \ref{theorem1} we know that,
\begin{equation}
\exists t_0>0,~ \ni |\sin(\theta)|< min\bigg(\dfrac{\delta}{B},|\sin(\theta_0)|\bigg),~\forall t>t_0,
\label{t0 definition}
\end{equation}

This implies that $||\xi(t)||_{\infty}<\delta,~\forall t>t_0$, 

where $\delta:=min~((a_2/2)-b_1,a_1)$.

Using Lemma \ref{lemma_sigma}, we can then conclude that the error dynamics $e_x$ and $e_v$ operate in the linear region of $\sigma_1$ and $\sigma_2$ after a finite time $t_0+t_2$. Further, if $||\hat{f}||_2$ is bounded away from zero, the reference trajectory $q_d$ is well defined and its derivatives are bounded. Therefore, the trajectories of (\ref{quad}) remain bounded in $(0,t_0+t_2)$.

In the linear region, the dynamics of $e_v$ can be written as:
\begin{equation}
m\dot{e}_v=-k_xe_x-k_ve_v+\mathcal{E}
\label{evdotlinear}
\end{equation} 

We choose a Lyapunov function candidate for the translational dynamics as
\begin{equation}
V_2:=\dfrac{1}{2}k_x||e_x||_2^2+\dfrac{1}{2}m||e_v ||_2^2 +ce_x^Te_v
\end{equation}Its derivative along the flow of (\ref{evdotlinear}) is obtained as
\begin{equation}
\dot{V}_2=-(k_v-c)|| e_v||_2^2-\frac{ck_x}{m}||e_x||_2^2-\frac{ck_v}{m}+\mathcal{E}^T\bigg(\frac{c}{m}e_x+e_v\bigg)
\label{v2dot1}
\end{equation} From (\ref{t0 definition}) we observe that,
\begin{equation}
||((q_d^Tq)q-q_d)||_2<1,~ \forall t>t_0
\end{equation}

Further, from the identity given in the proof of Lemma \ref{eq_leq_psi}, we observe that
\begin{eqnarray}
&&||((q_d^Tq)q-q_d)||_2\leq 2||e_q||_2  
\end{eqnarray}
This can be substituted in (\ref{v2dot1}) to obtain:
\begin{small}
\begin{eqnarray}
\dot{V}_2\leq  -(k_v(1-\sin(\theta))-c )||e_v||_2^2-\frac{ck_x}{m}(1-\sin(\theta))||e_x||^2_2 \nonumber \\
+\frac{ck_v}{m}(1+\sin(\theta))||e_x||_2||e_v||_2 \label{v2dotfinal} \\
+2||e_q||_2\bigg(k_x||e_x||_2||e_v||_2+\frac{c}{m}||f_d||_2||e_x||_2+||f_d||_2||e_v||_2\bigg)\nonumber
\end{eqnarray}
\end{small}In the linear region of the saturation functions, we can bound the cubic term in the above equation as: 
\begin{equation}
||e_q||_2k_x||e_x||_2||e_v||_2\leq a_1||e_q||_2||e_v||_2.
\label{cubicbound}
\end{equation}

Consider a Lyapunov function candidate for the complete dynamics as
\begin{equation}
V=V_1+V_2
\label{V}
\end{equation}where $V_1$ is defined as in (\ref{v1}).

Define $z_1=[||e_x||_2,||e_v||_2]^T$ and $z_2=[||e_q||_2,||e_{\Omega}||_2]^T$.

We can bound $V$ between two quadratic forms using Lemma \ref{eq_leq_psi} as:
\begin{equation}
z_1^TM_1z_1+z_2^TM_2z_2\leq V\leq z_1^TM_3z_1+z_2^TM_4z_2
\end{equation}where
\begin{equation}
\begin{matrix}
M_1=\frac{1}{2}\begin{bmatrix}
k_x & -c \\
-c & m
\end{bmatrix}, ~M_3=\frac{1}{2}\begin{bmatrix}
k_x & c \\
c & m
\end{bmatrix},\\ \\
M_2=\frac{1}{2}\begin{bmatrix}
2\alpha & 0 \\
0 & 1
\end{bmatrix},~M_4=\frac{1}{2}\begin{bmatrix}
4\alpha & 0 \\
0 & 1
\end{bmatrix}
\end{matrix}
\end{equation}From (\ref{v1dotfinal}), (\ref{v2dotfinal}),  (\ref{cubicbound}) and (\ref{V}), the derivative of $V$ can be obtained as:
\begin{equation}
\dot{V}\leq -Q(e_x,e_v,e_q,e_{\Omega})+\epsilon
\end{equation}

where 
\begin{equation}
Q(e_x,e_v,e_q,e_{\Omega})=z_1^TW_1z_1-z_1^TW_2z_2+z_2^TW_3z_2
\end{equation}

and
$W_3=\begin{bmatrix}
\alpha k_q & 0 \\
0 & k_{\Omega}
\end{bmatrix}$, $W_1$ and $W_2$ are as given in (\ref{wmatrices}).

Using the conditions in (\ref{condition on theorem2}), we observe that $ Q$ is a positive definite quadratic form and $V$ is sandwiched between two positive definite quadratic forms.
Hence, after a finite time $t_0+t_2$, the errors $[e_x,~e_v,~e_q,~e_{\Omega}]$ exponentially converge to an arbitrarily small open neighborhood of the origin.

 \end{proof}

\textbf{Remark 1:} By using saturated thrust feedback, it was possible to bound the error $\mathcal{E}$ in the translational dynamics, by sufficiently decreasing the reduced attitude error in the initial phase $t<t_0$. This was essential in order to ensure that the position errors decrease into the linear region, in finite time. This also allowed us to to bound the cubic term as (\ref{cubicbound}), which resulted in exponential stability. In \cite{tleeasme}, the authors restrict the stability analysis of the translational dynamics to a domain where $e_x$ is bounded. However, in order to remain within this domain the total system errors need to be further bounded, rendering the overall stability only local. In \cite{tlee}, the authors attempt to bound the velocity error $e_v$. However, such analysis is valid only when the gain $k_x$ is uniformly zero, failing which there can be no tractable bound on $e_v$. The stability analysis with the proposed control law in this paper is not restricted by any such conditions. 

 \textbf{Remark 2:} The conditions of the theorem dictate that the attitude tracking gains need to be high enough so that the translational errors enter the linear region of the saturation function. The limits chosen for the saturation functions are quite conservative, to ensure that the commanded thrust vector $\hat{f}$ is bounded away from the origin. This ensures that $\dot{q}_d$ and $\ddot{q}_d$ are bounded, thereby bounding the required torque. In practice however, one may further relax this limit within rotor thrust saturation.

 \section{Numerical Simulations}
Simulations were carried out on a variable pitch quadrotor which is capable of negative thrust, and a conventional quadrotor with positive rotor thrust constraint. Subsequent to rotor failure, the quadrotor was required to track a figure-of-8 trajectory while initially recovering from a downward facing pose. 
\subsection{Variable Pitch Quadrotor} 
 
The parameters of the quadrotor chosen for simulation are $m=1kg,~\omega_i=600 ~rad/s$ and $b_L=3.2\times 10^{-6}$, and the inertia matrix as

$J=\begin{bmatrix}
0.0972   & 0.0194   & 0.0195 \\
    0.0194   & 0.0974   & 0.0317 \\
    0.0195   & 0.0317 &   0.1584
\end{bmatrix}$. 

The nominal inertial matrix for control design was chosen as 
$J_0=diag(0.081,0.0812,0.1320)$. The propeller inertia was chosen as $J_r=5\times 10^{-5}$, and the rotational drag coefficient matrix was chosen as $K_d=diag(0.7,0.7,1.4)\times 10^{-4}$. 

The control gains were chosen as $k_q=8,~k_{\Omega}=10,~k_x=2,~k_v=3,~\Delta=3\times 10^{-3},~tol=10^{-3}$.

The reference position trajectory was chosen as a figure of '8' curve at constant altitude i.e.

$x_d(t)=2[\sin(2t),\cos(2t),5]^T$

The initial conditions were chosen as $x(0)=[5,5,5]$, $\dot{x}(0)=0$, $\Omega(0)=0$, and an initial orientation as a $140^{\circ}$ rotation about the $x ~\mathrm{axis}$ as:

$R(0)=\begin{bmatrix}
1 & 0 & 0 \\
0 & \cos(140^{\circ}) & -\sin(140^{\circ})\\
0 & \sin(140^{\circ}) & \cos(140^{\circ})
\end{bmatrix}$

 \begin{figure}[h]
\includegraphics[width=0.5\textwidth]{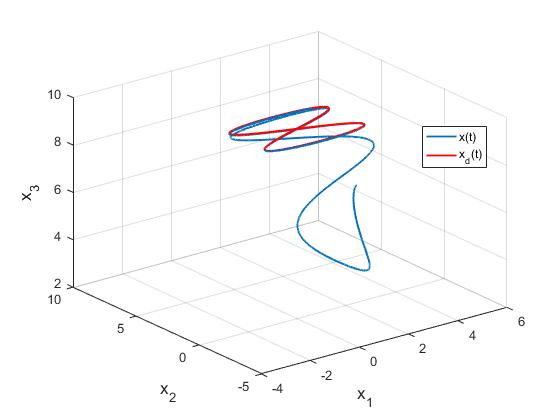}
\caption{Trajectory tracking after recovering from downward facing pose}
\label{fig:3d_1}
\end{figure}

 \begin{figure}[h]
\includegraphics[width=0.5\textwidth]{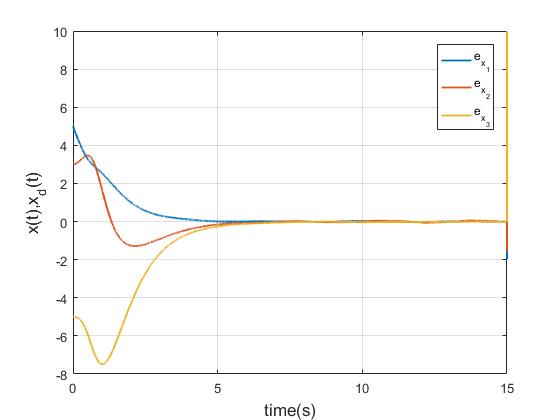}
\caption{Position error $e_x$ during the maneuver}
\label{fig:ex_1}
\end{figure}

 Fig.\ref{fig:3d_1} shows the quadrotor tracking a 'figure-of-8' reference trajectory after recovering from an inverted pose. Initially, when the reduced attitude error is large, there is a transient deviation from the reference trajectory. This can be seen in the position error plot in  Fig.\ref{fig:ex_1}. From this plot, it can also be seen that the tracking errors exponentially decrease and are bounded within an arbitrarily small open ball.

  \begin{figure}[h]
\includegraphics[width=0.5\textwidth]{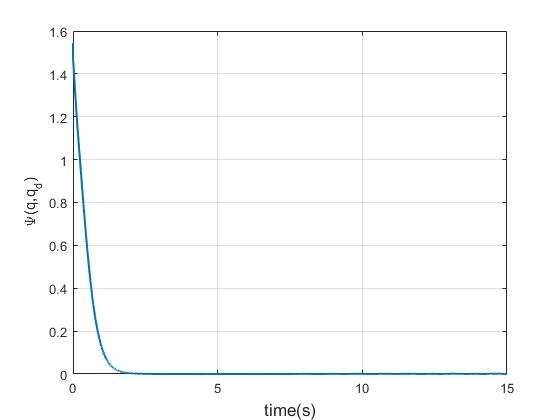}
\caption{Reduced attitude error $\Psi(q,q_d)$ during the maneuver}
\label{fig:psi_1}
\end{figure}
   
   Fig.\ref{fig:psi_1} shows the evolution of the reduced attitude error function $\Psi(q,q_d)$ during the maneuver. It can be seen that $\Psi$ decreases exponentially to an arbitrarily small open ball.
  
  \begin{figure}[h]
\includegraphics[width=0.5\textwidth]{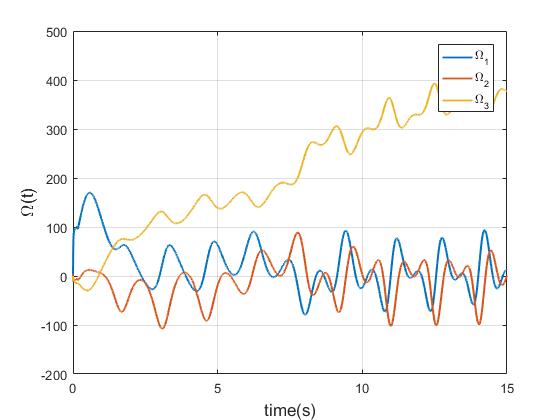}
\caption{Angular velocity $\Omega$ in $deg/sec$ during the maneuver}
\label{fig:omega_1}
\end{figure}  

Fig.\ref{fig:omega_1} shows the angular velocity about the three body axes during the maneuver. It can be seen that the angular velocity about the body $z$ axis increases rapidly and saturates due to rotational drag.

  \begin{figure}[h]
\includegraphics[width=0.5\textwidth]{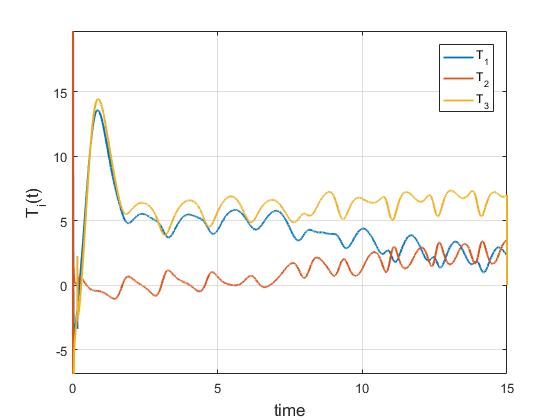}
\caption{Rotor thrusts $T_1,~T_2,~T_3$ during the maneuver}
\label{fig:thrust_1}
\end{figure} 

 Fig.\ref{fig:thrust_1} shows the variation of the thrust generated by the three functioning rotors during the maneuver. 

\subsection{Quadrotor with Positive Thrust Constraint}
The control law was simulated on a similar quadrotor where the rotor thrusts were constrained to be strictly positive. It was observed that when  $\Omega_3$ was sufficiently high, the control law was successfully able to execute the  attitude recovery and tracking maneuver. It was also observed that in case of large initial attitude errors, the tracking failed when the initial angular velocity $\Omega_3(0)$ was low. 
Similar parameters were chosen, except for a mass $m=3kg$ and initial angular velocity $\Omega_3(0)=2\pi ~rad/s$. 

 \begin{figure}[h]
\includegraphics[width=0.5\textwidth]{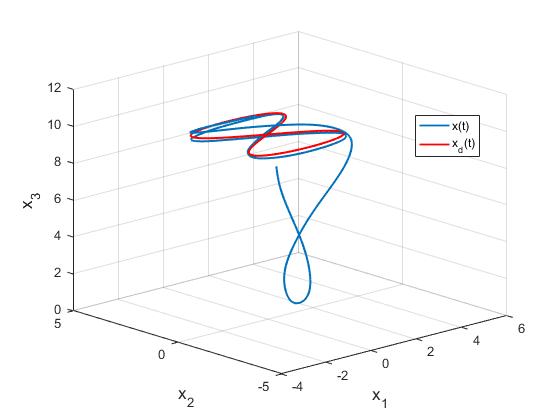}
\caption{Attitude recovery and position tracking with positive rotor thrust constraint}
\label{fig:3d_2}
\end{figure}

 \begin{figure}[h]
\includegraphics[width=0.5\textwidth]{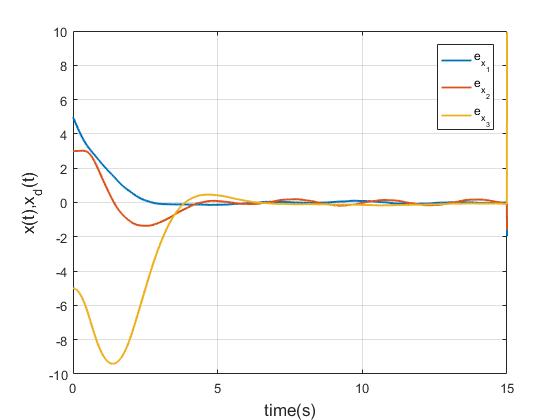}
\caption{Position error with positive rotor thrust constraint}
\label{fig:ex_2}
\end{figure}

Fig.\ref{fig:3d_2} and Fig.\ref{fig:ex_2} show larger transients  as compared with Fig.\ref{fig:3d_1} and Fig.\ref{fig:ex_1}, which is due to thrust saturation.

  \begin{figure}[h]
\includegraphics[width=0.5\textwidth]{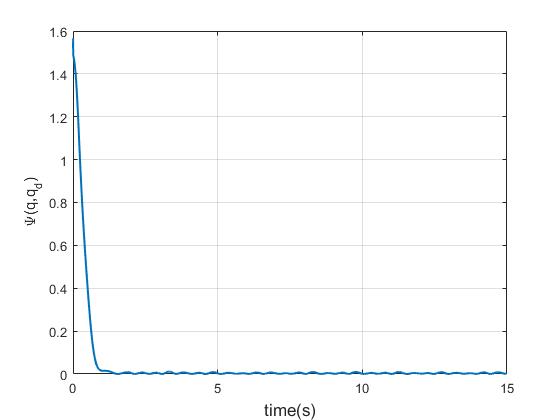}
\caption{Reduced attitude error $\Psi(q,q_d)$ with positive rotor thrust constraint}
\label{fig:psi_2}
\end{figure}
   
   Fig.\ref{fig:psi_2} shows fluctuations while stabilizing the reduced attitude, and a persistent error. This is due to the fact that when one rotor fails, the torque generated about one of the horizontal axes is strictly positive, which may lead to actuation error. The fluctuations and asymptotic errors can be further decreased by maintaining a higher $\Omega_3$ as shown in Fig.\ref{fig:omega_2}.
  
  \begin{figure}[h]
\includegraphics[width=0.5\textwidth]{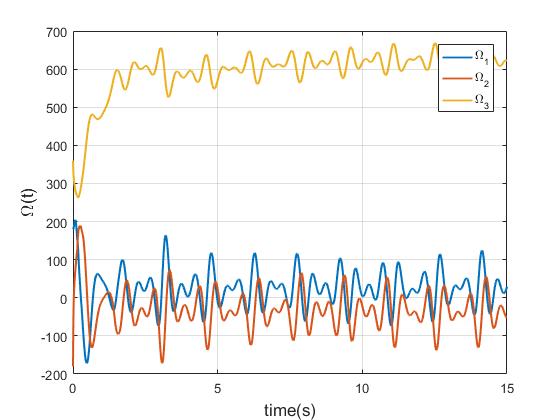}
\caption{Angular velocity $\Omega$ in $deg/sec$ with positive rotor thrust constraint}
\label{fig:omega_2}
\end{figure}  

  \begin{figure}[h]
\includegraphics[width=0.5\textwidth]{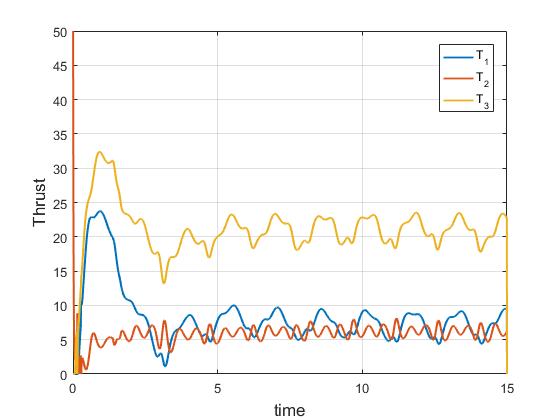}
\caption{Constrained rotor thrusts $T_1,~T_2,~T_3$ }
\label{fig:thrust_2}
\end{figure} 

Fig.\ref{fig:thrust_2} shows that the rotor thrusts operate within their constraints, and initially saturate when the attitude error is large. 

\textbf{Practical Considerations in
 Conventional Quadrotors:}

In conventional quadrotors, the rotor thrusts are constrained to be strictly positive and consequently the torque about one of the horizontal axes (say $U_1$) as well. Due to this, the controller performance can suffer due to large actuation error. A possible solution is to design a nominal trajectory $x_d(t)$ accounting for the initial conditions, such that $U_1$ is positive and uniformly bounded away from zero along this trajectory. This is possible considering that the position $x$ is a flat output of the dynamics on $SE(3)/SO(2)$. In \cite{andrea}, the authors discuss various periodic nominal trajectories satisfying the positive torque condition, about which they linearize the dynamics. From the exponential attractiveness of the geometric control law, it can be shown that if the tracking gains are appropriately chosen, the system trajectories will remain close to the nominal trajectory. 
As discussed in \cite{andrea}, such nominal trajectories require the angular velocity $\Omega_3$ to be significantly high. Post rotor failure, this angular velocity needs to be sufficiently increased before executing the maneuver. It is therefore essential to use high bandwidth attitude sensors (such as the MPU6050 DMP) and actuators. Further in conventional quadrotors, the tracking performance is improved if the ratio of the mass to inertia about the body $z$ axis, is sufficiently high. This ensures that in the initial phase where the angular velocity is increased, the translation errors do not grow significantly due to parasitic thrust. 

\section{Concluding Remarks}

We proposed a fault tolerant geometric control law for a quadrotor, subsequent to complete failure of a single rotor. It was demonstrated that unlike existing fault tolerant control laws, the quadrotor was able to perform aggressive maneuvers such as attitude recovery from an inverted pose and nontrivial trajectory tracking. This was primarily achieved by exploiting the geometric structure of the reduced configuration manifold, and designing a control law which was free of singularities which inhibit the performance envelop of the UAV. The back-stepping geometric control law for reduced attitude control also enabled reduced attitude tracking at arbitrarily high rates, which was essential for inhibiting transients and enhancing tracking performance.  While implementing the control law on a conventional quadrotor where the rotor thrusts are strictly positive, the angular rate about the body $z$ axis needs to be significantly high when the attitude error is large. Hence, when a fault in one of the rotors is detected, this angular rate needs to be first sufficiently increased before initiating the reduced attitude and position tracking maneuver.

\bibliographystyle{IEEEtran}
\bibliography{attitude_bib}

\end{document}